\def\DateTime{16/August/2014, 7:30 (Kyoto)}
\def\Version{Version $3.1$}
\def\yes{\if00}
\def\no{\if01}
\def\iftenpt{\no}
\def\ifelevenpt{\no}
\def\iftwelvept{\yes}

\def\ifquery{\yes}

\iftenpt
\documentclass[leqno]{amsart}
\else\fi
\ifelevenpt
\documentclass[leqno,11pt]{amsart}
\else\fi
\iftwelvept
\documentclass[leqno,12pt]{amsart}
\else\fi

\usepackage{amssymb}
\usepackage{amscd}
\usepackage{verbatim}
\usepackage{mathrsfs}

\usepackage[sc]{mathpazo} 
\usepackage[T1]{fontenc}

\ifelevenpt
\else\fi

\iftwelvept
\else\fi

\theoremstyle{plain}
\newtheorem{Theorem}{Theorem}[section]
\newtheorem{Proposition}[Theorem]{Proposition}
\newtheorem{Lemma}[Theorem]{Lemma}

\newtheorem{Corollary}[Theorem]{Corollary}

\newtheorem{Claim}{Claim}[Theorem]

\theoremstyle{definition}

\newtheorem{Remark}[Theorem]{Remark}
\newtheorem{Example}[Theorem]{Example}

\newtheorem{Question}[Theorem]{Question}

\renewcommand{\theequation}{\arabic{section}.\arabic{Theorem}.\arabic{Claim}}

\def\rom{\textup}
\newcommand{\ZZ}{{\mathbb{Z}}}
\newcommand{\QQ}{{\mathbb{Q}}}
\newcommand{\RR}{{\mathbb{R}}}
\newcommand{\KK}{{\mathbb{K}}}
\newcommand{\CC}{{\mathbb{C}}}
\newcommand{\PP}{{\mathbb{P}}}

\newcommand{\OO}{{\mathcal{O}}}

\newcommand{\DDD}{{\mathscr{D}}}

\newcommand{\XXX}{{\mathscr{X}}}
\newcommand{\Proj}{\operatorname{Proj}}

\newcommand{\Rat}{\operatorname{Rat}}

\newcommand{\Hom}{\operatorname{Hom}}

\newcommand{\Spec}{\operatorname{Spec}}

\newcommand{\Supp}{\operatorname{Supp}}

\newcommand{\mult}{\operatorname{mult}}

\newcommand{\rank}{\operatorname{rk}}
\newcommand{\ord}{\operatorname{ord}}

\newcommand{\adeg}{\widehat{\operatorname{deg}}}

\newcommand{\Sym}{\operatorname{Sym}}

\newcommand{\Nef}{\overline{\operatorname{Amp}}}
\newcommand{\Pef}{\overline{\operatorname{Eff}}}

\newcommand{\Div}{\operatorname{Div}}

\newcommand{\an}{\operatorname{an}}

\newcommand{\NS}{\operatorname{NS}}

\newcommand{\rest}[2]{\left.{#1}\right\vert_{{#2}}}
\ifquery
\def\query#1{\setlength\marginparwidth{65pt} 
\marginpar{\raggedright\fontsize{7.81}{9} 
\selectfont\upshape\hrule\smallskip 
#1\par\smallskip\hrule}} 

\else
\def\query#1{}
\fi

\begin{document}

\title[Toward a geometric analogue of Dirichlet's unit theorem]%
{Toward a geometric analogue \\ of Dirichlet's unit theorem}
\author{Atsushi Moriwaki}
\address{Department of Mathematics, Faculty of Science,
Kyoto University, Kyoto, 606-8502, Japan}
\email{moriwaki@math.kyoto-u.ac.jp}
\date{\DateTime, (\Version)}
\subjclass[2010]{Primary 14G15; Secondary 11G25, 11R04}
\begin{abstract}
In this article, we propose a geometric analogue of Dirichlet's unit theorem 
on arithmetic varieties \cite{MoD}, that is,
if $X$ is a normal projective variety over a finite field and $D$ is a pseudo-effective $\QQ$-Cartier divisor on $X$,
does it follow that $D$ is $\QQ$-effective?
We also give affirmative answers on an abelian variety and a projective bundle over a curve.
\end{abstract}


\maketitle

\section*{Introduction}

Let $K$ be a number field and $O_K$ the ring of integers in $K$.
Let $K(\CC)$ be the set of all embeddings $K \hookrightarrow \CC$.
For $\sigma \in K(\CC)$, the complex conjugation of $\sigma$ is denoted by
$\overline{\sigma}$, that is, $\overline{\sigma}(x) = \overline{\sigma(x)}\ (x \in K$).
Here we define $\Xi_K$ and $\Xi^0_K$ to be
\[
\begin{cases}
\Xi_K := \left\{ \xi \in \RR^{K(\CC)} \mid \xi(\sigma) = \xi(\overline{\sigma})\ (\forall \sigma) \right\}, \\[1ex]
\Xi^0_K := \left\{ \xi \in \Xi_K \mid \sum\nolimits_{\sigma \in K(\CC)} \xi(\sigma) = 0 \right\}.
\end{cases}
\]
The Dirichlet unit theorem asserts that the group $O_K^{\times}$ consisting of
units in $O_K$ is a finitely generated abelian group of rank $s :=\dim_{\RR} \Xi^0_K$.

Let us consider the homomorphism $L : K^{\times} \to \RR^{K(\CC)}$ given by
\[
L(x)(\sigma) := \log \vert \sigma(x) \vert \quad (x \in K^{\times},\ \sigma \in K(\CC)).
\]
It is easy to see the following:
\begin{enumerate}
\renewcommand{\labelenumi}{(\roman{enumi})}
\item
For a compact set $B$ in $\RR^{K(\CC)}$,
the set $\{ x \in O_K^{\times} \mid L(x) \in B \}$ is finite.

\item
$L : K^{\times} \to \RR^{K(\CC)}$ 
 extends to
$L_{\RR} : K^{\times} \otimes \RR \to \RR^{K(\CC)}$.

\item
$L_{\RR}: O_K^{\times} \otimes \RR \to \RR^{K(\CC)}$ is injective.

\item
$L_{\RR}(O_K^{\times} \otimes \RR) \subseteq  \Xi^0_K$.
\end{enumerate}
By using (i) and (iii),  we can see that
$O_K^{\times}$  is a finitely generated abelian group. 
The most essential part of the Dirichlet unit theorem is to show that
$O_K^{\times}$ is of rank $s$,
which is equivalent to see that, for any $\xi \in \Xi^0_K$,
there is $x \in O_K^{\times} \otimes \RR$ with $L_{\RR}(x) = \xi$.

In order to understand the equality $L_{\RR}(x) = \xi$ in terms of Arakelov geometry,
let us introduce several notations for arithmetic divisors on the arithmetic curve $\Spec(O_K)$.
An arithmetic $\RR$-divisor on $\Spec(O_K)$ is a pair $(D, \xi)$ consisting of an $\RR$-divisor $D$ on $\Spec(O_K)$ and $\xi \in \Xi_K$.
We often denote the pair $(D, \xi)$ by $\overline{D}$.
The arithmetic principal $\RR$-divisor $\widehat{(x)}_{\RR}$ of  $x \in K^{\times} \otimes \RR$ is the arithmetic $\RR$-divisor given by
\[
\widehat{(x)}_{\RR} := \left( \sum\nolimits_{P} \ord_P(x) [P], -2L_{\RR}(x) \right),
\]
where $P$ runs over the set of all maximal ideals of $O_K$ and
\[
\ord_P(x) := a_1\ord_P(x_1) + \cdots + a_r \ord_P(x_r)
\]
for $x = x_1^{a_1} \cdots x_r^{a_r}$ ($x_1, \ldots, x_r \in K^{\times}$ and $a_1, \ldots, a_r \in \RR$).
The arithmetic degree $\adeg(\overline{D})$ of an arithmetic $\RR$-divisor $\overline{D}= \left(\sum_{P} a_P [P], \xi\right)$ is defined to be
\[
\adeg(\overline{D}) := \sum_{P} a_P \log \#(O_K/P) + \frac{1}{2}\sum_{\sigma \in K(\CC)} \xi(\sigma).
\]
Note that
\[
\adeg\left(\widehat{(x)}_{\RR}\right) = 0\quad (x \in K^{\times} \otimes \RR)
\]
by virtue of the product formula.
Further, $\overline{D} = \left(\sum_{P} a_P [P], \xi\right)$ is said to be effective if $a_P \geq 0$ for all $P$ and $\xi(\sigma) \geq 0$ for all $\sigma$.

In \cite[SubSection~3.4]{MoD}, we proved the following:
\renewcommand{\theequation}{\arabic{section}.\arabic{Theorem}}%
\addtocounter{Theorem}{1}
\begin{equation}
\label{intro:Dirichlet:prop}
\text{``If $\adeg(\overline{D}) \geq 0$, then $\overline{D} + \widehat{(x)}_{\RR}$ is effective for some $x \in K^{\times} \otimes \RR$.''}
\end{equation}
\renewcommand{\theequation}{\arabic{section}.\arabic{Theorem}.\arabic{Claim}}%
This implies
the essential part of the Dirichlet unit theorem.
Indeed, we set $\overline{D}=(0,\xi)$ for $\xi \in \Xi_K^0$.
As $\adeg(\overline{D}) = 0$, by  the assertion \eqref{intro:Dirichlet:prop},
$\overline{D} + \widehat{(y)}_{\RR}$ is effective for some
$y \in K^{\times} \otimes \RR$, 
and hence
$\overline{D} + \widehat{(y)}_{\RR} = (0,0)$ because $\adeg(\overline{D}+ \widehat{(y)}_{\RR}) = 0$.
Here we set $y = u_1^{a_1} \cdots u_r^{a_r}$ such that
$u_1, \ldots, u_r \in K^{\times}$, $a_1, \ldots, a_r \in \RR$ and
$a_1, \ldots, a_r$ are linearly independent over $\QQ$.
By using the linear independency of $a_1, \ldots, a_r$ over $\QQ$, $\ord_P(y) = 0$ implies 
$\ord_P(u_i) = 0$ for all  maximal ideals $P$ of $O_K$ and $i=1, \ldots, r$, that is,
$u_i \in O_K^{\times}$ for $i=1, \ldots, r$.
Therefore, $\xi = L_{\RR}(y^2)$ and $y \in O_K^{\times} \otimes \RR$, as required.
In this sense, the above property
\eqref{intro:Dirichlet:prop} is an Arakelov theoretic interpretation of the classical Dirichlet unit theorem.

In \cite{MoD} and \cite{MoAdel},  we considered 
a  higher dimensional analogue of \eqref{intro:Dirichlet:prop}.
In the higher dimensional case, the condition ``$\adeg(\overline{D}) \geq 0$'' should be replaced by
the pseudo-effectivity of $\overline{D}$. Of course, this analogue is not true in general (cf. \cite{ChMoDsys}).
It is however a very interesting problem to find a sufficient condition for the existence of an arithmetic small $\RR$-section, that is,
an element $x$ such that
\[
x = x_1^{a_1} \cdots x_r^{a_r}\quad (\text{$x_1, \ldots, x_r$ are rational functions and $a_1, \ldots, a_r \in \RR$})
\]
and $\overline{D} + \widehat{(x)}_{\RR}$ is effective.
For example, in \cite{MoD} and \cite{MoAdel}, 
we proved that if $D$ is numerically trivial
and $\overline{D}$ is pseudo-effective,
then $\overline{D}$ has an arithmetic small $\RR$-section.
In this paper, we would like to consider a geometric analogue of the Dirichlet unit theorem in the above sense.

\bigskip
Let $X$ be a normal projective variety over an algebraically closed field $k$.
Let $\Div(X)$ denote the group of Cartier divisors on $X$.
Let $\KK$ be either the field $\QQ$ of rational numbers  or the field $\RR$ of real numbers. 
We define $\Div(X)_{\KK}$ to be $\Div(X)_{\KK} := \Div(X) \otimes_{\ZZ} \KK$, whose element is
called a {\em $\KK$-Cartier divisor} on $X$.
For $\KK$-Cartier divisors $D_1$ and $D_2$,
we say that $D_1$ is {\em $\KK$-linearly equivalent} to $D_2$, which is denoted by $D_1 \sim_{\KK} D_2$,
if there are non-zero rational functions $\phi_1, \ldots, \phi_r$ on $X$ and $a_1, \ldots, a_r \in \KK$
such that 
\[
D_1 - D_2 = a_1(\phi_1) + \cdots + a_r (\phi_r).
\]
Let $D$ be a $\KK$-Cartier divisor on $X$.
We say that $D$ is {\em big} if there is an ample $\QQ$-Cartier
divisor $A$ on $X$ such that $D - A$ is $\KK$-linearly
equivalent to an effective $\KK$-Cartier divisor. Further,
$D$ is said to be {\em pseudo-effective} if $D + B$ is big for
any big $\KK$-Cartier divisor $B$ on $X$.
Note that if $D$ is {\em $\KK$-effective}
(i.e.  $D$ is $\KK$-linearly equivalent to an effective $\KK$-Cartier divisor),
then $D$ is pseudo-effective.
The converse of the above statement holds on toric varieties (for example, \cite[Proposition~4.9]{BMPS}).
However, it is not true in general.
In the case where $k$ is uncountable  (for example, $k = \CC$), several examples are known such as
non-torsion numerically trivial invertible sheaves and 
Mumford's example on a minimal ruled surface (cf. \cite[Chapter~1, Example~10.6]{HartsAmple} and
\cite{MB}).
Nevertheless, we would like to propose the following question:

\begin{Question}[$\KK$-version]
\label{question:pseudo:effective}
We assume that $k$ is an algebraic closure of a finite field. 
If a $\KK$-Cartier divisor $D$ on $X$ is
pseudo-effective,
does it follow that $D$ is $\KK$-effective?
\end{Question}

This question is a geometric analogue of  the fundamental question introduced in \cite{MoD}.
In this sense, it turns out to be a geometric Dirichlet's unit theorem if it is true,
so that we often say that a $\KK$-Cartier divisor $D$ has {\em the Dirichlet property}
if $D$ is $\KK$-effective.
Note that the $\RR$-version implies
the $\QQ$-version (cf. Proposition~\ref{prop:R:effective:Q:effective}).
Moreover, the $\RR$-version does not hold in general.
In Example~\ref{exam:R:version:not:true}, we give an example, so that,
for the $\RR$-version, the question should be
\[
\text{``Under what conditions does it follow that $D$ is $\KK$-effective?''.}
\]
Further, the $\QQ$-version implies the following question due to Keel (cf. \cite[Question~0.9]{Keel} and
Remark~\ref{rem:G:Dirichlet:imply:Keel}).
The similar arguments on an algebraic surface are discussed
in the recent article by Langer \cite[Conjecture~1.7$\sim$1.9 and Lemma~1.10]{Langer2}.

\begin{Question}[S. Keel]
\label{question:Keel}
We assume that $k$ is an algebraic closure of a finite field and
$X$ is an algebraic surface over $k$. 
Let $D$ be a Cartier divisor on $X$.
If $(D \cdot C) > 0$ for all irreducible curves $C$ on $X$, 
is $D$ ample?
\end{Question}

By virtue of the Zariski decomposition, Question~\ref{question:pseudo:effective} on an algebraic surface is equivalent to ask the following:
\[
\text{``If $D$ is nef, then is $D$ $\KK$-effective?''}.
\]
One might expect that $D$ is semiample
(cf. \cite[Question 0.8.2]{Keel}).
However, Totaro \cite[Theorem~6.1]{To} found a Cartier divisor $D$ on an algebraic surface over a finite
field such that $D$ is nef but not semiample.
Totaro's example does not give a counter example of our question because 
we assert only  the $\QQ$-effectivity in Question~\ref{question:pseudo:effective}.

Inspired by the paper \cite{BS} due to Biswas and Subramanian,
we have the following partial answer to the above question.

\begin{Theorem}
\label{thm:main:pseudo:eff:eff}
We assume that $k$ is an algebraic closure of a finite field.
Let $C$ be a smooth projective curve over $k$ and let $E$ be a locally free sheaf of rank $r$ on $C$.
Let $\PP(E)$ be the projective bundle of $E$, that is,
$\PP(E) := \Proj\left(\bigoplus_{m=0}^{\infty} \Sym^m(E)\right)$.
If $D$ is a pseudo-effective $\KK$-Cartier divisor on $\PP(E)$, then $D$ is $\KK$-effective.
\end{Theorem}

In addition to the above result,
we can also give an affirmative answer for the $\QQ$-version of Question~\ref{question:pseudo:effective}
on abelian varieties.

\begin{Proposition}
\label{prop:pseudo:abelian:variety}
We assume that $k$ is an algebraic closure of a finite field.
Let $A$ be an abelian variety over $k$.
If $D$ is a pseudo-effective $\QQ$-Cartier divisor on $A$, then $D$ is $\QQ$-effective.
\end{Proposition}

Finally I would like to thank Prof. Biswas, Prof. Keel, Prof. Langer,
Prof. Tanaka and  Prof. Totaro for
their helpful comments.
Especially I would like to express my hearty thanks to Prof. Yuan for his nice example.
I also would like to thank the referee for the suggestions.

\section{Preliminaries}
Let $k$ be an algebraic closed field. 
Let $C$ be a smooth projective curve over $k$ and
let $E$ be a locally free sheaf of rank $r$ on $C$.
The projective bundle $\PP(E)$ of $E$ is given by
\[
\PP(E) := \Proj\left(\bigoplus_{m=0}^{\infty} \Sym^m(E)\right).
\]
The canonical morphism $\PP(E) \to C$ is denoted by $f_E$.
A tautological divisor $\Theta_E$ on $\PP(E)$ is a Cartier divisor on $\PP(E)$
such that $\OO_{\PP(E)}(\Theta_E)$ is isomorphic to the tautological invertible
sheaf $\OO_{\PP(E)}(1)$ on $\PP(E)$.
We say that $E$ is {\em strongly semistable} if,
for any surjective morphism $\pi : C' \to C$ of smooth projective curves,
$\pi^*(E)$ is semistable.
By definition, if $E$ is strongly semistable  and $\pi : C' \to C$ is
a surjective morphism of smooth projective curves over $k$, then
$\pi^*(E)$ is also strongly semistable.
A filtration
\[
0 = E_0 \subsetneq E_1 \subsetneq E_2 \subsetneq \cdots \subsetneq E_{s-1} \subsetneq E_{s} = E
\]
of $E$ is called the {\em strong Harder-Narasimham filtration} if
\[
\mu(E_1/E_0) > \mu(E_2/E_1) > \cdots > \mu(E_{s-1}/E_{s-2}) > \mu(E_{s}/E_{s-1})
\]
and
$E_{i}/E_{i-1}$ is a strongly semistable locally free sheaf on $C$ for each $i=1, \ldots, s$.
Recall the following well-known facts 
(F\ref{well:fact:01})--(F\ref{well:fact:05}) on strong semistability.
\begin{enumerate}
\renewcommand{\labelenumi}{(F\arabic{enumi})}
\item \label{well:fact:01}
A locally free sheaf $E$ on $C$ is strong semistable if and only if $\Theta_E - f_E^*(\xi_E/r)$ is nef,
where $\xi_E$ is a Cartier divisor on $C$ with $\OO_C(\xi_E) \simeq \det(E)$
(for example, see \cite[Proposition~7.1, (3)]{MoRBogo}). 

\item \label{well:fact:02}
Let $\pi : C' \to C$ be a surjective morphism of smooth projective curves over $k$ such that
the function field of $C'$ is a separable extension field over the function field of $C$. If $E$ is semistable, then $\pi^*(E)$
is also semistable 
(for example, see \cite[Proposition~7.1, (1)]{MoRBogo}).
In particular,
if $\operatorname{char}(k) = 0$, then
$E$ is strongly semistable if and only if
$E$ is semistable.
Moreover, in the case where $\operatorname{char}(k) > 0$, $E$ is strongly semistable if and only if
$(F^m)^*(E)$ is semistable for all $m \geq 0$,
where $F : C \to C$ is the absolute Frobenius map and
\[
F^m = \overbrace{F \circ \cdots \circ F}^{m}.
\]

\item \label{well:fact:03}
If $E$ and $G$ are strongly semistable locally free sheaves on $C$, then
$\Sym^m(E)$ and $E \otimes G$ are also strongly semistable
for all $m \geq 1$ (for example, see \cite[Theorem~7.2 and Corollary~7.3]{MoRBogo}).

\item \label{well:fact:04}
There is a surjective morphism $\pi : C' \to C$ of smooth projective curves over $k$ such that
$\pi^*(E)$ has the strong Harder-Narasimham filtration
(cf. \cite[Theorem~7.2]{Langer}).

\item \label{well:fact:05}
We assume that $k$ is an algebraic closure of a finite field.
If $E$ is a strongly semistable locally free sheaf 
on $C$ with $\det(E) \simeq \OO_C$,
then there is a surjective morphism $\pi : C' \to C$ of smooth projective curves over $k$
such that $\pi^*(E) \simeq \OO_{C'}^{\oplus \rank E}$ 
(cf. \cite[p. 557]{BH}, \cite[Theorem~3.2]{Su} and \cite{BS}).
\end{enumerate}

\bigskip
The purpose of this section is to prove the following characterizations of
pseudo-effective $\RR$-Cartier divisors and nef $\RR$-Cartier divisors on $\PP(E)$.
This is essentially due to Nakayama \cite[Lemma~3.7]{Nakayama} in which he works over
the complex number field.

\begin{Proposition}
\label{prop:criterion:nef:pseudo}
We assume that $E$ has the strong Harder-Narasimham filtration:
\[
0 = E_0 \subsetneq E_1 \subsetneq E_2 \subsetneq \cdots \subsetneq E_{s-1} \subsetneq E_{s} = E.
\]
Then, for an $\RR$-divisor $A$ on $C$, we have the following:

\begin{enumerate}
\renewcommand{\labelenumi}{(\arabic{enumi})}
\item
$\Theta_E - f^*(A)$ is pseudo-effective if and only if $\deg(A) \leq \mu(E_1)$.

\item
$\Theta_E - f^*(A)$ is nef if and only if $\deg(A) \leq \mu(E/E_{s-1})$.
\end{enumerate}
\end{Proposition}

Let us begin with the following lemma.

\begin{Lemma}
\label{lem:vansihing:sym}
We assume that 
$E$ has a filtration 
\[
0 = E_0 \subsetneq E_1 \subsetneq \cdots \subsetneq E_{s-1} \subsetneq E_s = E
\]
such that
$E_i/E_{i-1}$ is a strongly semistable locally free sheaf on $C$ and
$\deg(E_i/E_{i-1}) < 0$ for all $i=1, \ldots, s$.
Then,
$H^0(C, \Sym^m(E) \otimes G) = 0$ for $m \geq 1$ and
a strongly semistable locally free sheaf $G$ on $C$ with $\deg(G) \leq 0$.
\end{Lemma}

\begin{proof}
We prove it by induction on $s$.
In the case where $s = 1$, $E$ is strongly semistable and $\deg(E) < 0$, so that
$\Sym^m(E) \otimes G$ is also strongly semistable by (F\ref{well:fact:03}) and 
\[
\deg(\Sym^m(E) \otimes G) < 0.
\]
Therefore, $H^0(C, \Sym^m(E) \otimes G) = 0$.

Here we assume that $s > 1$. Let us consider an exact sequence
\[
 0 \to E_{s-1} \to E \to E/E_{s-1} \to 0.
\]
By \cite[Chapter~II, Exercise~5.16, (c)]{Harts}, there is a filtration 
\[
\Sym^m(E) = F^0 \supsetneq F^1 \supsetneq \cdots \supsetneq F^{m} \supsetneq F^{m+1} = 0 
\]
such that 
\[
F^j/F^{j+1} \simeq \Sym^{j}(E_{s-1}) \otimes \Sym^{m-j}(E/E_{s-1})
\]
for each $j = 0, \ldots, m$. 
By using the hypothesis of induction, 
\[
H^0(C, (F^{j}/F^{j+1}) \otimes G) = 0
\]
for $j=1, \ldots, m$
because $\Sym^{m-j}(E/E_{s-1}) \otimes G$ is strongly semistable by (F\ref{well:fact:03}) and 
\[
\deg(\Sym^{m-j}(E/E_{s-1}) \otimes G) \leq 0.
\]
Moreover, since $\Sym^{m}(E/E_{s-1}) \otimes G$ is strongly semistable by (F\ref{well:fact:03}) and 
\[
\deg(\Sym^{m}(E/E_{s-1}) \otimes G) < 0,
\]
we have
\[
H^0(C, (F^{0}/F^{1}) \otimes G) = H^0(C, \Sym^{m}(E/E_{s-1}) \otimes G) = 0.
\]
Therefore, by using an exact sequence 
\[
0 \to F^{j+1} \otimes G \to F^{j} \otimes G \to (F^{j}/F^{j+1}) \otimes G \to 0,
\]
we have
\[
H^0(C, F^{j+1} \otimes G) \overset{\sim}{\longrightarrow} H^0(C, F^{j} \otimes G)
\]
for $j=0, \ldots, m$,
which implies that 
$H^0(C, \Sym^m(E) \otimes G) = 0$,
as required.
\end{proof}

\begin{proof}[Proof of Proposition~\ref{prop:criterion:nef:pseudo}]
It is sufficient to show the following:

\begin{enumerate}
\renewcommand{\labelenumi}{(\alph{enumi})}
\item
If $A$ is a $\QQ$-Cartier divisor and $\deg(A) < \mu(E_1)$, then $\Theta_E - f^*(A)$ is $\QQ$-effective.

\item
If $A$ is a $\QQ$-Cartier divisor and $\deg(A) > \mu(E_1)$, then $\Theta_E - f^*(A)$ is not pseudo-effective.

\item
If $\Theta_E - f^*(A)$ is nef, then $\deg(A) \leq \mu(E/E_{s-1})$.

\item
If $\Theta_E - f^*(A)$ is not nef, then $\deg(A) > \mu(E/E_{s-1})$.
\end{enumerate}

\medskip
(a) 
Let $\theta$ be a divisor on $C$ with $\deg(\theta) = 1$.
As $E_1$ is strongly semistable, by (F\ref{well:fact:01}),
$\Theta_{E_1} - \mu(E_1)f_{E_1}^*(\theta)$ is nef, so that
we can see that $\Theta_{E_1} - f_{E_1}^*(A)$ is nef and big because
\[
\Theta_{E_1} - \deg(A)f_{E_1}^*(\theta) = \Theta_{E_1} - \mu(E_1) f_{E_1}^*(\theta) + (\mu(E_1) - \deg(A))f_{E_1}^*(\theta).
\]
Therefore, there is a positive integer $m_1$ such that
$m_1 A$ is a divisor on $C$ and 
\[
H^0\left(\PP(E_1), \OO_{\PP(E_1)}(m_1\Theta_{E_1} - f_{E_1}^*(m_1 A))\right) \not= 0.
\]
In addition,
\begin{align*}
H^0\left(\PP(E_1), \OO_{\PP(E_1)}(m_1\Theta_{E_1} - f_{E_1}^*(m_1A))\right) & = H^0(C, \Sym^{m_1}(E_1) \otimes \OO_{C}(-m_1 A)) \\
& \subseteq H^0(C, \Sym^{m_1}(E) \otimes \OO_{C}(-m_1 A)) \\
& = H^0\left(\PP(E), \OO_{\PP(E)}(m_1\Theta_{E} - f_E^*(m_1 A))\right),
\end{align*}
so that $\Theta_{E} - f_E^*(A)$ is $\QQ$-effective.

\medskip
(b) 
Let $B$ be an ample $\QQ$-divisor on $C$ with $\deg(B) < \deg(A) - \mu(E_1)$.
Let $\pi : C' \to C$ be a surjective morphism of smooth projective curves over $k$
such that $\pi^*(-A+B)$ is a Cartier divisor on $C'$.
Note that 
\[
\mu\left(\pi^*(E_i/E_{i-1}) \otimes \OO_{C'}(\pi^*(-A + B))\right) < 0
\]
for $i=1, \ldots, s$, and hence, by Lemma~\ref{lem:vansihing:sym},
\[
H^0\left(C', \Sym^m(\pi^*(E)) \otimes \OO_{C'}(m\pi^*(-A + B)))\right) = 0
\]
for all $m \geq 1$. In particular, if $b$ is a positive integer such that 
$b(-A + B)$ is a Cartier divisor,
then 
\[
H^0\left(C, \Sym^{mb}(E) \otimes \OO_{C}(mb(-A + B)))\right) = 0
\]
for $m \geq 1$.
Here we assume that $\Theta_E - f_E^*(A)$ is pseudo-effective.
Let $a$ be a positive integer such that $\Theta_E - f_E^*(A) + af_E^*(B)$ is ample.
Then
\[
(a-1)(\Theta_E - f_E^*(A)) + \Theta_E - f_E^*(A) + af_E^*(B) = a(\Theta_E + f_E^*(-A + B))
\]
is big, so that we can find a positive integer $m_1$ such that
\begin{multline*}
H^0\left(C, \Sym^{m_1ab}(E) \otimes \OO_C(m_1ab(-A + B))\right) \\
=
H^0\left(\PP(E), \OO_{\PP(E)}(m_1ab(\Theta_E + f_E^*(-A + B)))\right) \not= 0,
\end{multline*}
which is a contradiction.

\medskip
(c) 
Note that 
\[
\PP(E/E_{s-1}) \subseteq \PP(E),\quad
\Theta_{E/E_{s-1}} \sim \rest{\Theta_{E}}{\PP(E/E_{s-1})}\quad\text{and}\quad
f_{E/E_{s-1}} = \rest{f_E}{\PP(E/E_{s-1})},
\] 
so that
$\Theta_{E/E_{s-1}} - f_{E/E_{s-1}}^*(A)$ is nef on 
$\PP(E/E_{s-1})$. 
Let $\xi_{E/E_{s-1}}$ be a Cartier divisor on $C$ with $\OO_C(\xi_{E/E_{s-1}}) \simeq \det(E/E_{s-1})$.
If we set $e = \rank E/E_{s-1}$ and $G = \xi_{E/E_{s-1}}/e  - A$,
then 
\[
\Theta_{E/E_{s-1}} - f_{E/E_{s-1}}^*(A) = \Theta_{E/E_{s-1}} - f_{E/E_{s-1}}^*(\xi_{E/E_{s-1}}/e) + f_{E/E_{s-1}}^*(G).
\]
Since $\Theta_{E/E_{s-1}} - f_{E/E_{s-1}}^*(\xi_{E/E_{s-1}}/e)$ is nef by (F\ref{well:fact:01}) and
\[
\left( \Theta_{E/E_{s-1}} - f_{E/E_{s-1}}^*(\xi_{E/E_{s-1}}/e) \right)^e = 0,
\]
we have
\[
0 \leq \left( \Theta_{E/E_{s-1}} - f_{E/E_{s-1}}^*(A) \right)^e = e \deg(G).
\]
Therefore, $\deg(G) \geq 0$, and hence $\deg(A) \leq \mu(E/E_{s-1})$.

\medskip
(d) We can find an irreducible curve $C_0$ of $X$ such that
$( \Theta_{E} - f_E^*(A) \cdot C_0) < 0$. Clearly $C_0$ is flat over $C$.
Let $C_1$ be the normalization of $C_0$ and $h : C_1 \to C$ the induced morphism.
Let us consider the following commutative diagram:
\[
\begin{CD}
\PP(E) @<{\PP(h)}<< \PP(h^*(E)) \\
@V{f_E}VV @VV{f_{h^*(E)}}V \\
C @<{h}<< C_1
\end{CD}
\]
Note that $\PP(h)^*(\Theta_{E} - f_E^*(A)) \sim_{\RR} \Theta_{h^*(E)} - f_{h^*(E)}^*(h^*(A))$.
Further, there is a section $S$ of $f_{h^*(E)}$ such that $\PP(h)_*(S) = C_0$.
Let $Q$ be the quotient line bundle of $h^*(E)$ corresponding to the section $S$.
As
\[
0 = h^*(E_0) \subsetneq h^*(E_1) \subsetneq h^*(E_2) \subsetneq \cdots \subsetneq h^*(E_{s-1}) \subsetneq h^*(E_{s}) = h^*(E)
\]
is the Harder-Narasimham filtration of $h^*(E)$, we can easily see
\[
\deg(Q) \geq \mu(h^*(E/E_{s-1})) = \deg(h) \mu(E/E_{s-1}).
\]
On the other hand,
\begin{align*}
\deg(Q) - \deg(h)\deg(A) = 
(\Theta_{h^*(E)} - f_{h^*(E)}^*(h^*(A)) \cdot S) = (\Theta_{E} - f_E^*(A) \cdot C_0) < 0,
\end{align*}
and hence $\mu(E/E_{s-1}) < \deg(A)$.
\end{proof}

\bigskip
Finally let us consider the following three results.

\begin{Lemma}
\label{lem:eff:finite:covering}
Let $\KK$ be either $\QQ$ or $\RR$.
Let $\mu : X' \to X$ be a generically finite morphism of normal projective varieties over $k$.
For a $\KK$-Cartier divisor $D$ on $X$, $D$ is $\KK$-effective 
if and only if $\mu^*(D)$ is $\KK$-effective.
\end{Lemma}

\begin{proof}
Clearly, if $D$ is $\KK$-effective, 
then $\mu^*(D)$ is $\KK$-effective.
Let $K$ and $K'$ be the function fields of $X$ and $X'$, respectively.
Here we assume that $\mu^*(D)$ is $\KK$-effective,
that is,
there are $\phi'_1, \ldots, \phi'_r \in {K'}^{\times}$ and $a_1, \ldots, a_r \in \KK$
such that
$\mu^*(D) + a_1(\phi'_1) + \cdots + a_r (\phi'_r)$ is effective, so that
\[
\mu_*\left( \mu^*(D) + a_1(\phi'_1) + \cdots + a_r (\phi'_r) \right) = 
\deg(\mu)D + a_1 \mu_*((\phi'_1)) + \cdots + a_r \mu_*((\phi'_r))
\]
is effective.
Note that $\mu_*(\phi'_i) = (N_{K'/K}(\phi'_i))$ (cf. \cite[Proposition~1.4]{Fulton}),
where $N_{K'/K}$ is the norm map of $K'$ over $K$, and hence
\[
D + (a_1/\deg(\mu)) (N_{K'/K}(\phi'_1)) + \cdots + (a_r/\deg(\mu)) (N_{K'/K}(\phi'_r))
\]
is effective.
Therefore, $D$ is $\KK$-effective.
\end{proof}

\begin{Lemma}
\label{lem:num:zero:div}
Let $\KK$ be either $\QQ$ or $\RR$.
We assume that $k$ is an algebraic closure of a finite field.
Let $X$ be a normal projective variety over $k$ and
$D$ a $\KK$-Cartier divisor on $X$.
If $D$ is numerically trivial, then $D$ is $\KK$-linearly equivalent to the zero divisor.
\end{Lemma}

\begin{proof}
If $\KK = \QQ$, then the assertion is well-known, so that we assume that $\KK = \RR$.
We set $D = a_1 D_1 + \cdots + a_r D_r$, where $D_1, \ldots, D_r$ are Cartier divisors on $X$
and $a_1, \ldots, a_r \in \RR$. Considering a $\QQ$-basis of $\QQ a_1 + \cdots + \QQ a_r$ in $\RR$,
we may assume that $a_1, \ldots, a_r$ are linearly independent over $\QQ$.
Let $C$ be an irreducible curve on $X$. Note that
\[
0 = (D \cdot C) = a_1(D_1 \cdot C) + \cdots + a_r(D_r \cdot C)
\]
and $(D_1 \cdot C), \ldots, (D_r \cdot C) \in \ZZ$, and hence
$(D_1 \cdot C) = \cdots = (D_r \cdot C) = 0$
because
$a_1, \ldots, a_r$ are linearly independent over $\QQ$.
Thus $D_1, \ldots, D_r$ are numerically equivalent to zero, so that
$D_1, \ldots, D_r$ are $\QQ$-linearly equivalent to the zero divisor.
Therefore, the assertion follows.
\end{proof}

\begin{Proposition}
\label{prop:R:effective:Q:effective}
Let $X$ be a normal projective variety over $k$ and
let $D$ be a $\QQ$-Cartier divisor on $X$.
If $D$ is $\RR$-effective, then $D$ is $\QQ$-effective.
\end{Proposition}

\begin{proof}
As $D$ is $\RR$-effective,
there are non-zero rational functions $\psi_1, \ldots, \psi_l$ on $X$ and
$b_1, \ldots, b_l \in \RR$ 
such that
$D + b_1 (\psi_1) + \cdots + b_l (\psi_l)$
is effective. We set $V = \QQ b_1 + \cdots + \QQ b_l \subseteq \RR$.
If $V \subseteq \QQ$, then $b_1, \ldots, b_l \in \QQ$,
so that we may assume that $V \not\subseteq \QQ$.

\begin{Claim}
There are non-zero rational functions $\phi_1, \ldots, \phi_r$ on $X$,
$a_1, \ldots, a_r \in \RR$ and a $\QQ$-Cartier divisor $D'$ on $X$ such that
$D \sim_{\QQ} D'$, $D' + a_1(\phi_1) + \cdots + a_r(\phi_r)$ is effective and
$1, a_1, \ldots, a_r$ are linearly independent over $\QQ$.
\end{Claim}

\begin{proof}
We can find a basis $a_1, \ldots, a_{r}$
of $V$ over $\QQ$ with the following properties:
\begin{enumerate}
\renewcommand{\labelenumi}{(\roman{enumi})}
\item
If we set $b_i = \sum_{j=1}^{r} c_{ij} a_j$, then
$c_{ij} \in \ZZ$ for all $i,j$.

\item
If $V \cap \QQ \not= \{ 0 \}$,
then $a_1 \in \QQ^{\times}$.
\end{enumerate}
We put $\phi_j = \prod_{i=1}^l \psi_i^{c_{ij}}$. Note that
$\sum_{i=1}^l b_i (\psi_i) = \sum_{j=1}^{r} a_j ( \phi_j )$.
Therefore, in the case where $V \cap \QQ = \{ 0 \}$, $1, a_1, \ldots, a_r$ are
linearly independent over $\QQ$ and $D + \sum_{j=1}^{r} a_j ( \phi_j )$ is effective.
Otherwise, $1, a_2, \ldots, a_r$ are linearly independent over $\QQ$ and
$\left(D + a_1(\phi_1)\right) + \sum_{j=2}^{r} a_j ( \phi_j )$ is effective.
\end{proof}

We set $L = D' + a_1 (\phi_1) + \cdots + a_r (\phi_r)$.
Let $\Gamma$ be a prime divisor with $\Gamma \not\subseteq \Supp(L)$. Then
\[
0 = \mult_{\Gamma}(L) = \mult_{\Gamma} (D') + a_1 \ord_{\Gamma}(\phi_1) + \cdots + a_r \ord_{\Gamma}(\phi_r),
\]
so that 
$\mult_{\Gamma} (D') = \ord_{\Gamma}(\phi_1) = \cdots = \ord_{\Gamma}(\phi_r) = 0$
because
$1, a_1, \ldots, a_r$ are linearly independent over $\QQ$.
Thus,
\[
\Supp(D'), \Supp((\phi_1)), \ldots, \Supp((\phi_r)) \subseteq \Supp\left( L \right).
\]
Therefore, we can find $a'_1, \ldots, a'_r \in \QQ$ 
such that
$D' + a'_1 (\phi_1) + \cdots + a'_r (\phi_r)$ is effective, and hence $D$ is $\QQ$-effective.
\end{proof}

\section{Proof of Theorem~\ref{thm:main:pseudo:eff:eff}}

Let $k$ be an algebraic closure of a finite field. 
Let $C$ be a smooth projective curve over $k$.
Let us begin with the following lemma.

\begin{Lemma}
\label{lem:deg:non:negative:div}
Let $\KK$ be either $\QQ$ or $\RR$.
Let $A$ be a $\KK$-Cartier divisor on $C$.
If $\deg(A) \geq 0$, then $A$ is $\KK$-effective.
\end{Lemma}

\begin{proof}
If $\KK=\QQ$, then the assertion is obvious.
We assume that $\KK = \RR$.
If $\deg(A) = 0$, the assertion follows from
Lemma~\ref{lem:num:zero:div}.
Next we consider the case where $\deg(A) > 0$.
We can find a $\QQ$-Cartier divisor $A'$ such that $A' \leq A$ and $\deg(A') > 0$.
Thus the previous observation implies the assertion.
\end{proof}

As a consequence of (F\ref{well:fact:03}), (F\ref{well:fact:04}) and (F\ref{well:fact:05}), 
we have the following splitting 
theorem, which was obtained by Biswas and Parameswaran \cite[Proposition~2.1]{BP}.

\begin{Theorem}
\label{thm:splitting:locally:free:sheaf}
For a locally free sheaf $E$ on $C$,
there are a surjective morphism $\pi : C' \to C$ of smooth projective curves over $k$ and
invertible sheaves $L_1, \ldots, L_r$ on $C'$ such that
$\pi^*(E) \simeq L_1 \oplus \cdots \oplus L_r$.
\end{Theorem}

\begin{proof}
For reader's convenience, we give a sketch of the proof.
First we assume that $E$ is strongly semistable.
Let $\xi_E$ be a Cartier divisor on $C$ with $\OO_C(\xi_E) \simeq \det(E)$.
Let $h : B \to C$ be
a surjective morphism of smooth projective curves over $k$
such that $h^*(\xi_E)$ is divisible by $\rank(E)$.
We set $E' = h^*(E) \otimes \OO_{B}(-h^*(\xi_E)/\rank(E))$.
As $\det(E') \simeq \OO_{B}$, the assertion follows from 
(F\ref{well:fact:05}).

By the above observation, 
it is sufficient to find a surjective morphism $\pi : C' \to C$ of smooth projective curves over $k$ and
strongly semistable locally free sheaves $Q_1, \ldots, Q_n$ on $C'$ such that
\[
{\pi}^*(E) = Q_1 \oplus \cdots \oplus Q_n.
\]
Moreover, by (F\ref{well:fact:04}), we may assume that
$E$ has the strong Harder-Narasimham filtration
\[
0 = E_0 \subsetneq E_1 \subsetneq E_2 \subsetneq \cdots \subsetneq E_{n-1} \subsetneq E_{n} = E.
\]
Clearly we may further assume that $n \geq 2$. 
For a non-negative integer $m$,
we set 
\[
C_m := X \times_{\Spec(k)} \Spec(k),
\]
where the morphism  $\Spec(k) \to \Spec(k)$ is given by $x \mapsto x^{1/p^m}$.
Let $F_k^m : C_m \to C$ be the relative $m$-th Frobenius morphism over $k$.
Put
\[
G^m_{i,j} := (F_k^m)^*\left((E_{j}/E_i)\otimes (E_i/E_{i-1})^{\vee}\right) \otimes \omega_{C_m}
\]
for $i=1, \ldots, n-1$ and $j = i, \ldots, n$.
We can find a positive integer $m$ such that
\[
\mu\left(G^m_{i,i+1}\right) \\
=
p^m \left( \mu(E_{i+1}/E_i) - \mu(E_i/E_{i-1}) \right) + \deg(\omega_{C})< 0
\]
for all $i = 1, \ldots, n-1$.
By using (F\ref{well:fact:03}), we can see that
\[
0 = G^m_{i,i} \subsetneq G^m_{i,i+1} \subsetneq G^m_{i, i+2} \subsetneq \cdots \subsetneq G^m_{i, n-1} \subsetneq G^m_{i,n}
\]
is the strong Harder-Narasimham filtration of
$G^m_{i,n}$,
so that 
$H^0\left(C_m,\ G^m_{i,n}\right) = \{ 0 \}$,
which yields
\[
\operatorname{Ext}^1\left((F_k^m)^*(E/E_{i}), \ (F_k^m)^*(E_i/E_{i-1})\right) = 0
\]
because of Serre's duality theorem. Therefore, an exact sequence
\[
0 \to (F_k^m)^*(E_i/E_{i-1}) \to (F_k^m)^*(E/E_{i-1}) \to (F_k^m)^*(E/E_{i}) \to 0
\]
splits, that is, $(F_k^m)^*(E/E_{i-1}) \simeq (F_k^m)^*(E_i/E_{i-1}) \oplus (F_k^m)^*(E/E_{i})$ for $i=1, \ldots, n-1$, and hence
\[
(F_k^m)^*(E) \simeq \bigoplus_{i=1}^n (F_k^m)^*(E_i/E_{i-1}),
\]
as required.
\end{proof}

\begin{proof}[Proof of Theorem~\ref{thm:main:pseudo:eff:eff}]
By virtue of Theorem~\ref{thm:splitting:locally:free:sheaf} and Lemma~\ref{lem:eff:finite:covering},
we may assume that 
\[
E \simeq L_1 \oplus \cdots \oplus L_r
\]
for some invertible sheaves $L_1, \ldots, L_r$ on $C$.
We set 
\[
d = \max \{ \deg(L_1), \ldots, \deg(L_r) \}\quad\text{and}\quad
I = \{ i \mid \deg(L_i) = d \}.
\]
There is a $\KK$-Cartier divisor $A$ on $C$ such that
$D \sim_{\KK} \lambda \Theta_E - f_E^*(A)$ for some $\lambda \in \KK$. 
Let $M$ be an ample divisor on $C$ such that $T := \Theta_E + f_E^*(M)$ is ample.
As $D$ is pseudo-effective, we have 
\[
0 \leq (D \cdot T^{r-2} \cdot f_E^*(M)) = ((\lambda T - f_E^*(A + \lambda M)) \cdot T^{r-2} \cdot f_E^*(M)) = \lambda \deg(M),
\]
and hence
$\lambda \geq 0$. If $\lambda =0$, then $0 \leq (D \cdot T^{r-1}) = \deg(-A)$.
Thus, by Lemma~\ref{lem:deg:non:negative:div}, $-A$ is $\KK$-effective, so that
the assertion follows.

We assume that $\lambda > 0$. Replacing $D$ by $D/\lambda$, we may assume that $\lambda = 1$.
Let $\xi$ be a Cartier divisor on $C$ such that $\OO_C(\xi) \simeq L_{i_0}$ for some $i_0 \in I$.
Note that the first part $E_1$ of the strong Harder-Narasimham filtration
of $E$ is $\bigoplus_{i \in I} L_i$,
so that, by Proposition~\ref{prop:criterion:nef:pseudo},
$\deg(A) \leq \deg(\xi)$. 
If we set $B = \xi - A$,
then, by Lemma~\ref{lem:deg:non:negative:div}, 
$B$ is $\KK$-effective because $\deg(B) \geq 0$.
Moreover, as 
\[
\Theta_E - f_E^*(A) = \Theta_E - f_E^*(\xi) + f_E^*(B),
\]
it is sufficient to consider the case where 
$D = \Theta_E - f_E^*(\xi)$.
In this case, the assertion is obvious because
\[
H^0(\PP(E), \OO_{\PP(E)}(D)) = H^0(C, E \otimes \OO_{C}(-\xi)) = 
H^0\left(C, \ \bigoplus_{i=1}^r L_i \otimes \OO_{C}(-\xi) \right) \not= \{ 0 \}.
\]
\end{proof}

As a consequence of Theorem~\ref{thm:main:pseudo:eff:eff},
we can recover a result due to \cite{BS}.

\begin{Corollary}
\label{cor:pseudo:eff:eff}
Let $k$, $C$ and $E$ be same as in Theorem~\ref{thm:main:pseudo:eff:eff}.
We assume that $r = 2$.
Let $D$ be a Cartier divisor on $\PP(E)$ such that $(D \cdot Y) > 0$
for all irreducible curves $Y$ on $\PP(E)$. Then $D$ is ample.
\end{Corollary}

\begin{proof}
As $D$ is nef, $D$ is pseudo-effective, so that, by Theorem~\ref{thm:main:pseudo:eff:eff},
there is an effective $\QQ$-Cartier divisor $E$ on $X$ such that
$D \sim_{\QQ} E$. As $E \not= 0$, we have
$(D \cdot D) = (D \cdot E) > 0$.
Therefore, $D$ is ample by Nakai-Moishezon criterion.
\end{proof}

\begin{Remark}
\label{rem:G:Dirichlet:imply:Keel}
The argument in the proof of 
Corollary~\ref{cor:pseudo:eff:eff} actually
shows that the $\QQ$-version of
Question~\ref{question:pseudo:effective} on algebraic surfaces implies 
Question~\ref{question:Keel}.
\end{Remark}

\section{Numerically effectivity on abelian varieties}

The purpose of this section is to give an affirmative answer for
the $\QQ$-version of Question~\ref{question:pseudo:effective} on
abelian varieties.
Let $A$ be an abelian variety over an algebraically closed field $k$.
A key observation is the following proposition.

\begin{Proposition}
\label{prop:nef:num:Q:effective}
If a $\QQ$-Cartier divisor $D$ on $A$ is nef, then $D$ is numerically equivalent to a $\QQ$-effective $\QQ$-Cartier divisor.
\end{Proposition}

\begin{proof}
We prove it by induction on $\dim A$. If $\dim A \leq 1$, then the assertion is obvious.
Clearly we may assume that $D$ is a Cartier divisor, so that we set $L = \OO_A(D)$.
As $L \otimes [-1]^*(L)$ is numerically equivalent to $L^{\otimes 2}$ (cf. \cite[p.75, (iv)]{Mumford}),
we may assume that $L$ is symmetric, that is, $L \simeq [-1]^*(L)$.
Let $K(L)$ be the closed subgroup of $A$ given by 
$K(L) = \{ x \in A \mid T_x^*(L) \simeq L \}$
(cf. \cite[p.60, Definition]{Mumford}).
If $K(L)$ is finite, then $L$ is nef and big by virtue of
\cite[p.150, The Riemann-Roch theorem]{Mumford}, so that
$D$ is $\QQ$-effective.
Otherwise, let $B$ be the connected component of $K(L)$ containing $0$.

\begin{Claim}
\begin{enumerate}
\renewcommand{\labelenumi}{(\arabic{enumi})}
\item
$\rest{T^*_x(L)}{B} \simeq \rest{L}{B}$ for all $x \in A$.

\item
$\rest{L^{\otimes 2}}{B+ x} \simeq \OO_{B+x}$ for $x \in A$.
\end{enumerate}
\end{Claim}

\begin{proof}
(1) Let $N$ be an invertible sheaf on $A \times A$ given by
\[
N = m^*(L) \otimes p_1^*(L^{-1}) \otimes p_2^{*}(L^{-1}),
\]
where $p_i : A \times A \to A$ is the projection to the $i$-th factor ($i=1,2$) and
$m$ is the addition morphism.
Note that $\rest{N}{B \times A} \simeq \OO_{B \times A}$ (cf. \cite[p.123, \S13]{Mumford}).
Fixing $x \in A$,
let us consider a morphism $\alpha : B \to B \times A$ given by $\alpha(y) = (y, x)$. Then
\[
\OO_B \simeq \alpha^*\left( \rest{m^*(L) \otimes p_1^*(L^{-1}) \otimes p_2^*(L^{-1})}{B \times A} \right) \simeq
\rest{T_x^*(L)}{B} \otimes \rest{L^{-1}}{B},
\]
as required.

(2) First we consider the case where $x = 0$.
As $\rest{N}{B \times A} \simeq \OO_{B \times A}$, we have
$\rest{N}{B \times B} \simeq \OO_{B \times B}$.
Using a morphism $\beta : B \to B \times B$ given by $\beta(y) = (y, -y)$,
we have
\[
\OO_B \simeq \beta^*(\rest{N}{B \times B}) = \rest{L^{-1}}{B} \otimes \rest{[-1]^*(L^{-1})}{B} \simeq \rest{L^{\otimes -2}}{B},
\]
as required.

In general, for $x \in A$, by (1) and the previous observation
together with the following commutative diagram
\[
\begin{CD}
B + x @>>> A \\
@V{T_{-x}}VV @VV{T_{-x}}V \\
B @>>> A,
\end{CD}
\]
we can see
\begin{align*}
\OO_{B + x} & = T_{-x}^*(\OO_B) \simeq T_{-x}^*\left(\rest{L^{\otimes 2}}{B}\right)
\simeq T_{-x}^*\left(\rest{T^*_x(L)^{\otimes 2}}{B}\right) \\
& = T_{-x}^*\left(\rest{T^*_x(L^{\otimes 2})}{B}\right) = \rest{T_{-x}^*(T^*_x(L^{\otimes 2}))}{B+x} = \rest{L^{\otimes 2}}{B + x}.
\end{align*}
\end{proof}

Let $\pi : A \to A/B$ be the canonical homomorphism.
By (2) in the above claim,
\[
\dim_{k(y)} H^0\left(\pi^{-1}(y), L^{\otimes 2}\right) = 1
\]
for all $y \in A/B$, so that,
by \cite[p.51, Corollary~2]{Mumford},
$\pi_*(L^{\otimes 2})$ is an invertible sheaf on $A/B$ and $\pi_*(L^{\otimes 2}) \otimes k(y) 
\overset{\sim}{\longrightarrow} H^0(\pi^{-1}(y), L^{\otimes 2})$.
Therefore, the natural homomorphism 
$\pi^*(\pi_*(L^{\otimes 2})) \to L^{\otimes 2}$ is an isomorphism, that is,
there is a $\QQ$-Cartier divisor $D'$ on $A/B$ such that $\pi^*(D') \sim_{\QQ} D$.
Note that $D'$ is also nef, so that, by the hypothesis of induction,
$D'$ is numerically equivalent to a $\QQ$-effective $\QQ$-Cartier divisor, and
hence the assertion follows.
\end{proof}

\begin{proof}[Proof of Proposition~\ref{prop:pseudo:abelian:variety}]
Proposition~\ref{prop:pseudo:abelian:variety} is a consequence of
Lemma~\ref{lem:num:zero:div} and
Proposition~\ref{prop:nef:num:Q:effective} because a pseudo-effective
$\QQ$-Cartier divisor on an abelian variety is nef.
\end{proof}

\begin{Example}
\label{exam:R:version:not:true}
Here we show that
the $\RR$-version of Question~\ref{question:pseudo:effective} does not hold in general.
Let $k$ be an algebraically closed field ($k$ is not
necessarily an algebraic closure of a finite field).
Let $C$ be an elliptic curve over $k$ and $A := C \times C$.
Let $\NS(A)$ be the N\'{e}ron-Severi group of $A$.
Note that $\rho := \rank \NS(A) \geq 3$.
By using the Hodge index theorem, we can find a basis
$e_1, \ldots, e_{\rho}$ of $\NS(A)_{\QQ} := \NS(A) \otimes_{\ZZ} \QQ$ with
the following properties:
\begin{enumerate}
\renewcommand{\labelenumi}{(\arabic{enumi})}
\item
$e_1$ is the class of the divisor $\{ 0 \} \times C + C \times \{ 0 \}$.
In particular, $(e_1 \cdot e_1) = 2$.

\item
$(e_i \cdot e_i) < 0$ for all $i=2, \ldots, \rho$.

\item
$(e_i \cdot e_j) = 0$ for all $1 \leq i \not= j \leq \rho$.
\end{enumerate}
We set $\lambda_i := -(e_i \cdot e_i)$ for $i=2,\ldots,\rho$.
Let $\Nef(A)$ be the closed cone in $\NS(A)_{\RR} := \NS(A) \otimes_{\ZZ} \RR$ generated by
ample $\QQ$-Cartier divisors on $A$.
It is well known that
\begin{align*}
\Nef(A) & = \left\{ \xi \in \NS(A)_{\RR} \mid (\xi^2) \geq 0, \ (\xi \cdot e_1) \geq 0 \right\} \\
& =
\left\{ x_1 e_1 + x_2 e_2 + \cdots + x_{\rho} e_{\rho} \mid 
\lambda_2 x_2^2 + \cdots + \lambda_{\rho} x_{\rho}^2 \leq 2 x_1^2, \ x_1 \geq 0 \right\}.
\end{align*}
We choose $(a_2, \ldots, a_{\rho}) \in \RR^{\rho-1}$ such that
\[
(a_2, \ldots, a_{\rho}) \not\in \QQ^{\rho-1}\quad\text{and}\quad 
\lambda_2a_2^2 + \cdots + \lambda_{\rho} a_{\rho}^2 = 2.
\]
Let $E_i$ be a $\QQ$-Cartier divisor on $A$ such that
the class of $E_i$ in $\NS(A)_{\QQ}$ is equal to $e_i$ for $i=1, \ldots, \rho$.
If we set $D := E_1 + a_2 E_2 + \cdots + a_{\rho} E_{\rho}$, then we have the following claim,
which is sufficient for our purpose.

\begin{Claim}
$D$ is nef and $D$ is not numerically equivalent to an effective $\RR$-Cartier divisor.
\end{Claim}

\begin{proof}
Clearly $D$ is nef. If we set $e'_1 = e_1/\sqrt{2}$ and
$e'_i = e_i/\sqrt{\lambda_i}$ for $i=2, \ldots, \rho$, then
\[
\Nef(A) =
\left\{ y_1 e'_1 + y_2 e'_2 + \cdots + y_{\rho} e'_{\rho} \mid 
{y_2}^2 + \cdots + {y_{\rho}}^2 \leq {y_1}^2, \ y_1 \geq 0 \right\}.
\]
Therefore,
as $[D] \in \partial(\Nef(A)_{\RR})$,
we can choose 
\[
H \in \Hom_{\RR}(\NS(A)_{\RR}, \RR)
\]
such that
\[
\text{$H \geq 0$ on $\Nef(A)$}\quad\text{and}\quad
\{ H = 0 \} \cap \Nef(A) = \RR_{\geq 0} [D],
\]
where $[D]$ is the class of $D$ in $\NS(A)_{\RR}$.
We assume that $D$ is numerically equivalent to an effective $\RR$-Cartier divisor 
$c_1 \Gamma_1 + \cdots + c_r \Gamma_r$,
where $c_1, \ldots, c_r \in \RR_{\geq 0}$ and $\Gamma_1, \ldots, \Gamma_r$ are prime divisors on $A$.
As $[D] \not= 0$, we may assume that $c_1, \ldots, c_r \in \RR_{>0}$.
Note that $[\Gamma_1], \ldots, [\Gamma_r] \in \Nef(A)$ and
\[
0 = H([D]) = c_1 H([\Gamma_1]) + \cdots + c_r H([\Gamma_r]),
\]
so that $H([\Gamma_1]) = \cdots = H([\Gamma_r]) = 0$, and hence
$[\Gamma_1], \ldots, [\Gamma_r] \in \RR_{\geq 0}[D]$.
In particular, there is $t \in \RR_{\geq 0}$ with $[\Gamma_1] = t [D]$.
Here we can set 
\[
[\Gamma_1] = b_1 e_1 + \cdots + b_{\rho} e_{\rho} \quad (b_1, \ldots, b_{\rho} \in \QQ).
\]
Thus $b_1 = t$, $b_2 = t a_2, \ldots, b_{\rho} = t a_{\rho}$. 
As $[\Gamma_1] \not= 0$, $t \in \QQ^{\times}$, and hence
$(a_2, \ldots, a_{\rho}) = t^{-1}(b_2, \ldots, b_{\rho}) \in \QQ^{\rho-1}$. This is a contradiction.
\end{proof}
\end{Example}

\begin{Remark}
Let $k$ be an algebraic closure of a finite field and let
$X$ be a normal projective variety over $k$.
Let $\NS(X)$ be the N\'{e}ron-Severi group of $X$ and
$\NS(X)_{\RR} := \NS(X) \otimes_{\ZZ} \RR$.
Let $\Pef(X)$ be the closed cone in $\NS(X)_{\RR}$ generated by pseudo-effective $\RR$-Cartier divisors on $X$.
We assume that $\Pef(X)$ is a rational polyhedral cone, that is,
there are pseudo-effective $\QQ$-Cartier divisors $D_1, \ldots, D_n$ on $X$ such that
$\Pef(X)$ is generated by the classes of $D_1, \ldots, D_n$.
Then the $\QQ$-version of Question~\ref{question:pseudo:effective} implies 
the $\RR$-version of Question~\ref{question:pseudo:effective}.
\end{Remark}

\begin{Example}
This is an example due to Yuan \cite{Yuan}.
Let us fix an algebraically closed field $k$ and an integer $g \geq 2$.
Let $C$ be a smooth projective curve over $k$ and
$f : X \to C$ an abelian scheme over $C$ of relative dimension $g$.
Let $L$ be an $f$-ample invertible sheaf on $X$ such that $[-1]^*(L) \simeq L$ and 
$L$ is trivial along the zero section of $f : X \to C$.

\begin{Claim}
\begin{enumerate}
\renewcommand{\labelenumi}{(\arabic{enumi})}
\item
$[2]^*(L) \simeq L^{\otimes 4}$.

\item
$L$ is nef.
\end{enumerate}
\end{Claim}

\begin{proof}
(1) As $\rest{[2]^*(L)}{f^{-1}(x)} \simeq \rest{L^{\otimes 4}}{f^{-1}(x)}$ for all $x \in C$,
there is an invertible sheaf $M$ on $C$ such that $[2]^*(L) \simeq L^{\otimes 4} \otimes f^*(M)$.
Let $Z_0$ be the zero section of $f : X \to C$. Then
\[
\OO_{Z_0} \simeq [2]^*(\rest{L}{Z_0}) = \rest{[2]^*(L)}{Z_0} \simeq \rest{L^{\otimes 4} \otimes f^*(M)}{Z_0} \simeq M,
\]
so that we have the assertion.

(2) Let $A$ be an ample invertible sheaf on $C$ such that
$L \otimes f^*(A)$ is ample.
Let $\Delta$ be a horizontal curve on $X$. As $f \circ [2^n] = f$ and
$[2^n]^*(L) \simeq L^{\otimes 4^n}$ by using (1),
\[
0 \leq (L \otimes f^*(A) \cdot [2^n]_*(\Delta)) = ([2^n]^*(L \otimes f^*(A)) \cdot \Delta)
= (L^{\otimes 4^n} \otimes f^*(A) \cdot \Delta),
\]
so that $(L \cdot \Delta) \geq -4^{-n} (f^*(A) \cdot \Delta)$ for all $n > 0$.
Thus $(L \cdot \Delta) \geq 0$.
\end{proof}

\begin{Claim}
If the characteristic of $k$ is zero and $f$ is non-isotrivial, then
$L$ does not have the Dirichlet property \rom{(}i.e. $L$ is not $\QQ$-effective\rom{)}.
\end{Claim}

\begin{proof}
The following proof is due to Yuan \cite{Yuan}.
An alternative proof can be found in \cite[Theorem~4.3]{CoPi}.
We need to see that $H^0(X, L^{\otimes n}) = 0$ for all $n > 0$.
We set $d_n = \rank f_*(L^{\otimes n})$.
By changing the base $C$ if necessarily, we may assume that all $(d_n)^2$-torsion points on
the generic fiber $X_{\eta}$ of $f : X \to C$ are
defined over the function field of $C$. By using the algebraic theta theory due to
Mumford (especially \cite[the last line in page 81]{MumDefA}), there is an invertible sheaf $M$ on $C$
such that $f_*(L^{\otimes n}) = M^{\oplus d_n}$.
On the other hand, by \cite{MR}, 
\[
\deg(\det(f_*(L^{\otimes n}))^{\otimes 2} \otimes f_*(\omega_{X/C})^{\otimes d_n}) = 0,
\]
that is, $2 \deg(M) + \deg( f_*(\omega_{X/C})) = 0$.
As $f$ is  non-isotrivial, we can see that $\deg(f_*(\omega_{X/C})) > 0$, so that $\deg(M) < 0$,
and hence the assertion follows.
\end{proof}

If the characteristic of $k$ is positive, we do not know the $\QQ$-effectivity of $L$ in general.
In \cite{Moret}, there is an example with the following properties:
\begin{enumerate}
\renewcommand{\labelenumi}{(\arabic{enumi})}
\item
$g = 2$ and $C = \PP^1_k$.

\item
There are an abelian surface $A$ over $k$ and an isogeny $h : A \times \PP^1_k \to X$
over $\PP^1_k$.
\end{enumerate}

\begin{Claim}
In the above example, $L$ has the Dirichlet property.
\end{Claim}

\begin{proof}
Replacing $L$ by $L^{\otimes n}$, we may assume that $d := \rank f_*(L) > 0$.
Let 
\[
p_1 : A \times \PP^1_k \to A
\quad\text{and}\quad
p_2 : A \times \PP^1_k \to \PP^1_k
\]
be the projections to $A$ and $\PP^1_k$, respectively.
Note that $h^*(L)$ is symmetric and $h^*(L)$ is trivial along the zero section of 
$p_2$. Since $\omega_{A \times \PP^1_k/P^1_k} \simeq p_1^*(\omega_{A})$,
we have $(p_2)_*(\omega_{A \times \PP^1_k/P^1_k}) \simeq \OO_{\PP^1_k}$, so that, by \cite{MR},
$\deg(\det((p_2)_*(h^*(L)))) = 0$, that is,  if we set 
\[
(p_2)_*(h^*(L)) = \OO_{\PP^1_k}(a_1) \oplus \cdots \oplus \OO_{\PP^1_k}(a_d),
\]
then $a_1 + \cdots + a_d = 0$. Thus $a_i \geq 0$ for some $i$, and hence
\[
H^0(A \times \PP^1_k, h^*(L)) \not= 0.
\]
Therefore, $L$ is $\QQ$-effective by Lemma~\ref{lem:eff:finite:covering}.
\end{proof}

The above claim suggests that the set of preperiodic points of the map $[2] : X \to X$ 
is not dense
in the analytification $X^{\an}_v$ at any place $v$ of $\PP^1_k$ with
respect to the analytic topology (cf. \cite{ChMoDsys}).
\end{Example}

\end{document}